\documentclass{amsart}
\usepackage{tikz}
\usetikzlibrary{decorations.pathmorphing}
\tikzset{every node/.style={circle,fill, inner sep=0pt, minimum size=1.5mm}, every picture/.style={scale=.85}}

\usepackage[foot]{amsaddr}

\usepackage{hyperref}

\usepackage{amsmath,amssymb,amsthm}
\allowdisplaybreaks

\newcounter{tbox}

\newtheorem{lemma}{Lemma}
\newtheorem{corollary}{Corollary}
\newtheorem{theorem}{Theorem}

\newtheorem{claim}{Claim}

\newtheorem{thm}{Theorem}

	\title{Minimal induced subgraphs of the class of 2-connected non-Hamiltonian wheel-free graphs}

		\author{Aristotelis Chaniotis$^\ast$} 
		
		\author{Zishen Qu$^{\ast, \ddagger}$} 
		
		\author{Sophie Spirkl$^{\ast,\dagger}$}

		\address{$^{\dagger}$ {We acknowledge the support of the Natural Sciences and Engineering Research Council of Canada (NSERC), [funding reference number RGPIN-2020-03912]. Cette recherche a été financée par le Conseil de recherches en sciences naturelles et en génie du Canada (CRSNG), [numéro de référence RGPIN-2020-03912].}}
		\address{$^{\ast}$ Department of Combinatorics and Optimization, University of Waterloo, Waterloo, Ontario, N2L 3G1, Canada}
		\address{$^{\ddagger}$ Current address: Department of Mathematics, University of Illinois at Urbana-Champaign, Urbana, IL.}
\begin{document}
\maketitle
		
\begin{abstract}
Given a graph $G$ and a graph property $P$ we say that $G$ is minimal with respect to $P$ if no proper induced subgraph of $G$
has the property $P$. An HC-obstruction is a minimal 2-connected non-Hamiltonian graph. Given a graph $H$, a graph $G$ 
is $H$-free if $G$ has no induced subgraph isomorphic to $H$.
The main motivation for this paper originates from a theorem of 
Duffus, Gould, and Jacobson (1981), which characterizes all the minimal connected graphs with no Hamiltonian path.
In 1998, Brousek characterized all the claw-free HC-obstructions.
On a similar note, Chiba and Furuya (2021), 
characterized all (not only the minimal) 2-connected non-Hamiltonian
$\{K_{1,3}, N_{3,1,1}\}$-free graphs. Recently, Cheriyan, Hajebi, and two of us (2022), characterized 
all triangle-free HC-obstructions and all the HC-obstructions which are split graphs.
A wheel is a graph obtained from a cycle by adding a new vertex with at least three neighbors in the cycle. 
In this paper we characterize all the HC-obstructions which are wheel-free graphs.
\end{abstract}

\section{Introduction}
\label{sec.intro}

In this paper we consider finite undirected graphs without loops or multiple edges.

Let $G$ be a graph and let $X \subseteq V(G)$. We denote by $G[X]$ the induced subgraph of $G$ with vertex set $X$, 
and by $G \setminus X$ the graph $G[V(G)\setminus X]$.
Given two graphs $H$ and $G$, the graph $G$ is  {\em $H$-free} if it has no induced subgraph isomorphic to $H$.

Let $X \subseteq V(G)$.  The set $X$ is a {\em cutset} if $G \setminus X$ is not connected, and if $X$ is a cutset of cardinality $k$, then $X$
is a {\em $k$-cutset}. The set $X$ is a {\em clique}
if it is a set of pairwise adjacent vertices. If $X$ is both a clique and a cutset, then $X$ is a {\em clique cutset}. Let $F \subseteq E(G)$. The set $F$ is an 
{\em edge-cutset} if the graph $(V(G), E(G)\setminus F)$ is not connected. An edge-cutset of cardinality $k$ is a {\em $k$-edge-cutset}.

Let $v\in V(G)$ and let $X \subseteq V(G)$. 
We denote by $N_{X}(v)$ the set of neighbors of $v$ in $X$, by $d(v)$ the degree of the vertex $v$ in $G$, 
and by $\Delta(G)$ the maximum degree of a vertex in $G$. 
Given two graphs $G_{1}$ and $G_{2}$, we denote by $G_{1}\cup G_{2}$ the graph $(V(G_{1})\cup V(G_{2}), E(G_{1})\cup E(G_{2}))$.

A {\em Hamiltonian path} (resp.\ {\em Hamiltonian cycle}) in a graph $G$ is a (not necessarily induced) subgraph $H$ of $G$ which is
a path (resp.\ cycle), and $V(H) = V(G)$. A graph is  {\em Hamiltonian} if it has a Hamiltonian cycle and {\em non-Hamiltonian} otherwise.

Following the notation of \cite{2022.triangle-free.split-graphs.HC.obs}, 
we say that a graph $H$ is an {\em HP-obstruction} if $H$ is connected, has no Hamiltonian path, and every induced subgraph of
$H$ either equals $H$, is not connected, or has a Hamiltonian path. 
Analogously, a graph $H$ is an {\em HC-obstruction} if $H$ is 2-connected, has no Hamiltonian cycle, and every induced subgraph of $H$
either equals $H$, is not 2-connected, or is Hamiltonian.
	
Below we state some results which involve specific graphs not defined in the present paper. Our first motivation for this paper originates from the following result of
Duffus, Gould, and Jacobson \cite{1981.HP.obs.claw.net}, which characterizes all the graphs which are HP-obstructions.

\begin{thm}[{\cite{1981.HP.obs.claw.net}, see also {\cite[Theorem 2.9]{claw-free.HC-obs}}}]
There are exactly two HP-obstructions: the claw and the net.
\end{thm}

In 1998, Brousek \cite{brousek.claw-free} obtained a complete characterization of the claw-free HC-obstructions.
On a similar note, Chiba and Furuya \cite{chiba.furuya.not.only.minimal}, characterized all (not only the minimal) 2-connected non-Hamiltonian
$\{K_{1,3}, N_{3,1,1}\}$-free graphs.
Cheriyan, Hajebi, and two of us \cite{2022.triangle-free.split-graphs.HC.obs} characterized all HC-obstructions which are split graphs and all
HC-obstructions which are triangle-free graphs.

\begin{thm}[\cite{2022.triangle-free.split-graphs.HC.obs}]
The snare and all n-novae for $n\geq 2$ are HC-obstructions. Moreover, these are the only HC-obstructions which are split graphs.
\end{thm}

\begin{thm}[\cite{2022.triangle-free.split-graphs.HC.obs}] \label{thm:triangle-free}
All thetas, triangle-free closed thetas, and triangle-free wheels are HC-obstructions, and they are the only HC-obstructions which are triangle-free.
\end{thm}

We remark that the analogous problem with respect to the induced minor relation has been fully resolved by Ding and Marshall \cite{ding-marshall}, 
who obtained a complete characterization of the minimal, with respect to the induced minor relation, 2-connected non-Hamiltonian graphs.

A {\em wheel} is a graph obtained from a cycle $C$ by adding a new vertex (not in $V(C)$) which has at least three neighbors in $V(C)$.
We remark that in several papers about wheels (for example, in \cite{structure.only.prism}), the cycle $C$ is required to have at least four vertices and thus the above definition
is non-standard: for example, a four-vertex complete graph is a wheel according to our definition.

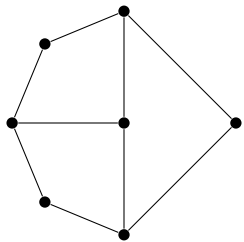
\begin{figure}[h]\centering
\begin{tikzpicture}
			\newdimen\R
			\R=1.75cm
			\node (h2) at (90:\R) {};
			\node (h5) at (270:\R) {};
			\node (ell) at (180:\R) {};
			\node (ell2) at (0:\R) {};
			\node (m1) at (135:\R) {};
			\node (m2) at (225:\R) {};
			\node (c) at (0,0) {};
			\draw (c) -- (h2);
			\draw (c) -- (h5);
			\draw (c) -- (ell);
			\draw (h2) -- (m1);
			\draw (m1) -- (ell);
			\draw (ell) -- (m2);
			\draw (m2) -- (h5);
			\draw (ell2) -- (h2);
			\draw (ell2) -- (h5);
			\end{tikzpicture}
\caption{An example of a wheel.}
\end{figure}

In this paper we give a complete characterization of the HC-obstructions which 
are wheel-free graphs. Since, by Theorem \ref{thm:triangle-free}, all triangle-free wheels are HC-obstructions, this is equivalent to characterizing all HC-obstructions which do not contain a wheel and contain a triangle. 

We proceed with some definitions in order to state our main result. Given two vertices, $u$ and $v$, a $(u,v)$-path is a path which has the vertices $u$ and $v$ as its ends. 
Given two sets of vertices, $A$ and $B$, an $(A,B)$-path is a $(u,v)$-path, where $u\in A$ and $v \in B$.

Let $G$ be a graph, let $K$ be the complete graph on $V(G)$, and let $F$ be a subset of $E(K)$ such that $F\cap E(G)=\emptyset$. 
We denote by $G + F$ the graph $(V(G), E(G)\cup F)$. Given a positive integer $k$, we denote by $[k]$ the set of positive integers $\{1,\ldots, k\}$.

Let $K$ and $L$ be two disjoint triangles on $\{k_{1},k_{2}, k_{3}\}$ and $\{l_{1}, l_{2}, l_{3}\}$ respectively.
For each $i \in [3]$ let $P_{i}$ be a $(k_{i}, l_{i})$-path, such that the paths $P_{1}, P_{2}, P_{3}$ are vertex-disjoint.
Let $P:= K \cup L \cup (\bigcup_{i\in [3]}P_{i})$. If for each $i \in [3]$ the path $P_{i}$ has length at least two, then $P$ is a {\em prism}, otherwise
$P$ is a {\em short prism}. 

Let $l$ be a vertex, such that $l \notin V(K)$. For each $i \in [3]$ let $P_{i}$ be a $(k_{i}, l)$-path of length at least two, 
such that the paths $P_{1}, P_{2}, P_{3}$ are internally vertex-disjoint. Then the graph $K \cup (\bigcup_{i\in [3]}P_{i})$ is  a {\em pyramid}.

Let $a$ and $b$ be two vertices. For each $i \in [3]$ let $P_{i}$ be an $(a, b)$-path of length at least two, such that the paths $P_{1}, P_{2}, P_{3}$ are internally vertex-disjoint.
The graph $\bigcup_{i\in [3]}P_{i}$ is  a {\em theta}.

		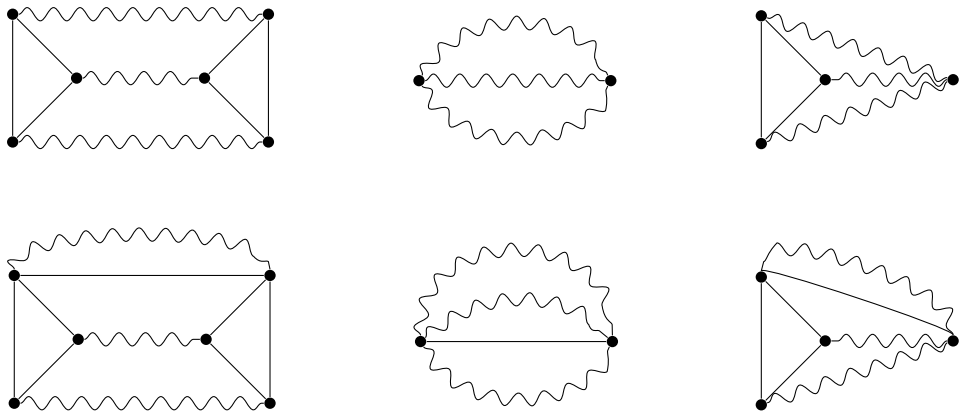
\begin{figure}[h!]\centering		
			\begin{tikzpicture}
			\node (k1) at (0,0) {};
			\node (k2) at (1,1) {};
			\node (k3) at (0,2) {};
			\node (s1) at (4,0) {};
			\node (s2) at (3,1)	{};
			\node (s3) at (4,2) {};
			
			\draw (k1) -- (k2);
			\draw (k2) -- (k3);	
			\draw (k1) -- (k3);
			\draw (s1) -- (s2);
			\draw (s2) -- (s3);	
			\draw (s1) -- (s3);	
			
			\draw[decorate, decoration=snake] (k3) -- (s3);
			\draw[decorate, decoration=snake] (k2) -- (s2);
			\draw[decorate, decoration=snake] (k1) -- (s1);
			\end{tikzpicture}
			\hfill
			\begin{tikzpicture}
		    \node (k1) at (0,1) {};
			\node (k2) at (3,1) {};
			\draw[decorate, decoration=snake] (k1) to[out=-60,in=-120] (k2);
			\draw[decorate, decoration=snake] (k1) to[out=60,in=120] (k2);
			\draw[decorate, decoration=snake] (k1) -- (k2);
		        
			\end{tikzpicture}
			\hfill
			\begin{tikzpicture}
			\node (k1) at (0,0) {};
			\node (k2) at (1,1) {};
			\node (k3) at (0,2) {};
			\node(v) at (3,1) {};
			\draw (k1) -- (k2);
			\draw (k2) -- (k3);	
			\draw (k1) -- (k3);
			\draw[decorate, decoration=snake] (v) -- (k1);
			\draw[decorate, decoration=snake] (v) -- (k2);
			\draw[decorate, decoration=snake] (v) -- (k3);
			\end{tikzpicture}
			
			\vspace{1cm}
			
			\begin{tikzpicture}
			\node (k1) at (0,0) {};
			\node (k2) at (1,1) {};
			\node (k3) at (0,2) {};
			\node (s1) at (4,0) {};
			\node (s2) at (3,1)	{};
			\node (s3) at (4,2) {};
			
			\draw (k1) -- (k2);
			\draw (k2) -- (k3);	
			\draw (k1) -- (k3);
			\draw (s1) -- (s2);
			\draw (s2) -- (s3);	
			\draw (s1) -- (s3);	
			
			\draw (k3) -- (s3);
			\draw[decorate, decoration=snake] (k2) -- (s2);
			\draw[decorate, decoration=snake] (k1) -- (s1);
			\draw[decorate, decoration=snake] (k3) .. controls +(up:0.75) and +(up:0.75) .. (s3);
			\end{tikzpicture}
			\hfill
			\begin{tikzpicture}
		    \node (k1) at (0,1) {};
			\node (k2) at (3,1) {};
			\draw[decorate, decoration=snake] (k1) to[out=-60,in=-120] (k2);
			\draw[decorate, decoration=snake] (k1) to[out=40,in=140] (k2);
			\draw (k1) -- (k2);
			
			\draw[decorate, decoration=snake] (k1) .. controls +(up:1.75) and +(up:1.75) .. (k2);
		        
			\end{tikzpicture}
			\hfill
			\begin{tikzpicture}
			\node (k1) at (0,0) {};
			\node (k2) at (1,1) {};
			\node (k3) at (0,2) {};
			\node(v) at (3,1) {};
			\draw (k1) -- (k2);
			\draw (k2) -- (k3);	
			\draw (k1) -- (k3);
			\draw (v) .. controls +(up:0.2) and +(up:0.2) ..  (k3);
			\draw[decorate, decoration=snake] (v) -- (k2);
			\draw[decorate, decoration=snake] (v) -- (k1);
			
			\draw[decorate, decoration=snake] (v) .. controls +(up:0.75) and +(up:0.75) .. (k3);
			\end{tikzpicture}

			\caption{Top row from left to right: A prism, a theta and a pyramid. Bottom row from left to right: 
			A prism$^{+}$, pyramid$^{+}$, theta$^{+}$. Squiggly edges represent paths of length at least two.}
			
		\end{figure}

A {\em prism$^{+}$/pyramid$^{+}$/theta$^{+}$ is a graph obtained from a prism/pyramid/theta}
by selecting one or more of the three paths involved in the definition
of the prism/pyramid/theta and adding for each of the selected paths an edge which joins its ends. (In the case of a theta$^+$, we add at most one edge between $a$ and $b$.)

A {\em 3-path-configuration} is a graph isomorphic to a prism, a pyramid, a theta, a prism$^{+}$, a pyramid$^{+}$ or a theta$^{+}$. We refer to 3-path-configuration as a
{\em 3PC}. Again, we note that our definitions differ from standard literature (such as \cite{structure.only.prism}): usually, only prisms, pyramids, and thetas are referred to as 3PCs, and usually short prisms, as well as pyramids with a path of length one, are included. Note that the latter are wheels according to our definition.

\medskip
In this paper we prove the following theorem:

\begin{theorem}
\label{th.main}
All 3PCs are HC-obstructions. 
Moreover, these are the only HC-obstructions which are wheel-free graphs.
\end{theorem}

\section{Wheel-free graphs}

In this section we prove \autoref{th.main}. In \cite{2022.triangle-free.split-graphs.HC.obs} it is proved that all thetas are HC-obstructions, by the exact same argument it follows that 
all thetas$^{+}$ are HC-obstructions.

\begin{lemma}[\cite{2022.triangle-free.split-graphs.HC.obs}]
\label{lem.obs.noHC.2conn.thetas}
All thetas and thetas$^{+}$ are HC-obstructions.
\end{lemma}

\begin{lemma}
\label{lem.obs.noHC.2conn.prism.pyramid}
All prisms, pyramids, prisms$^{+}$, pyramids$^{+}$ are HC-obstructions.
\end{lemma}

\begin{proof}
The fact that none of these graphs has a cutvertex can be easily checked. Also it can be easily checked that 
for each of these graphs, each of its proper induced subgraphs is either not 2-connected or  Hamiltonian.
Let us prove that each of the graphs in the statement of the lemma is non-Hamiltonian.

Let $G$ be a prism/prism$^{+}$/pyramid/pyramid$^{+}$ as in the definition in Section \ref{sec.intro}. 
Let suppose towards a contradiction that $G$ has a Hamiltonian cycle $C$. Then since for each 
$i\in [3]$, the path $P_{i}$ has at least one internal vertex which has degree two in $G$, the cycle $C$ contains as subgraph the path $P_{i}$ and 
if $G$ has an edge which joins the ends of $P_{i}$, then clearly $C$ does not contain this edge. 
Since $C$ is $2$-regular, it follows that it contains at most one of the edges of $G[k_{1}, k_{2}, k_{3}]$. Now one of the vertices $k_{1}, k_{2}, k_{3}$
has degree one in $C$, which is a contradiction.
\end{proof}

By \autoref{lem.obs.noHC.2conn.thetas} and \autoref{lem.obs.noHC.2conn.prism.pyramid}, we get the following:

\begin{corollary}
\label{cor.3pc.are.HCobs}
All 3PCs are HC-obstructions.
\end{corollary}

\medskip
In view of \autoref{cor.3pc.are.HCobs}, 
in order to prove \autoref{th.main}, it suffices to prove the following:

\begin{theorem}
\label{th.main.reduction}
If $G$ is a 2-connected, wheel-free graph which has no induced subgraph isomorphic to a 3PC, then $G$ is Hamiltonian.
\end{theorem}

Below we state some notions and results that we need for the proof of \autoref{th.main.reduction}.

If $H$ is a graph, then the {\em line graph} of $H$ is the graph $G$ whose vertices are the edges of $H$ and two vertices of $G$ are adjacent if and only if
the corresponding edges of $H$ share a vertex. Given a graph $H$ we denote its line graph by $L(H)$.
The following is well known: 

\begin{lemma}
\label{lem.line.graph.subgraph}
Let $H$ be a graph and $J$ be a subgraph of $H$. Then $L(J)$ is an induced subgraph of $L(H)$.
\end{lemma}

Following the notation of \cite{structure.only.prism}, we say that a graph $G$ is an {\em only-prism} graph if it is (theta, wheel, pyramid)-free.
A {\em short pyramid} is a graph as in the definition of the pyramid in Section \ref{sec.intro} with the difference that there exists exactly 
one $i\in [3]$ such that $P_{i}$ has length one.
We remark that our definitions of a wheel and a pyramid are sligthly different from the corresponding definitions in \cite{structure.only.prism}.
In particular, in \cite{structure.only.prism} the complete graph on four vertices is not considered a wheel and a {\em short pyramid} is considered as a pyramid, 
whereas we consider it to be a wheel. Overall, however, only-prism graphs in our sense are also only-prism graphs as defined in 
\cite{structure.only.prism}.

Let $G$ be a graph and $C$ be a cycle which is a subgraph of $G$. 
A {\em chord} of $C$ is an edge $e$ such that both the ends of $e$ are vertices of $C$ and $e \notin E(C)$.
A cycle $C$ is {\em chordless} if $C$ has no chords and a graph $G$ is {\em chordless} if every cycle of $G$ is chordless.

\begin{theorem}[\cite{structure.only.prism}]
\label{th.only.prism.str}
If $G$ is an only-prism graph, then either $G$ is the line graph of a triangle-free chordless graph, or $G$ admits a clique cutset.
\end{theorem}

We will use the following definitions from \cite{structure.biconn.chordless}. 
A graph $G$ is {\em 2-sparse} if every edge of $G$ is incident to at least one vertex of degree at most two.
A {\em proper 2-cutset} of a connected graph $G$ is a pair of non-adjacent vertices $u,v$
such that $V(G)$ can be partitioned into non-empty sets $X, Y$ and $\{u,v\}$ such that there is no edge between $X$ and $Y$, 
and with the property that each of the graphs $G[X\cup\{u,v\}]$ and $G[Y\cup\{u,v\}]$ 
contains a $(u,v)$-path, and neither of the graphs $G[X\cup\{u,v\}]$ and $G[Y\cup\{u,v\}]$ is a path. We say that 
$(X,Y,u,v)$ is a {\em split} of this proper 2-cutset.

\begin{theorem}[\cite{structure.biconn.chordless, structure.chordless.first.result}]
\label{th.structure.biconn.chordless}
If $H$ is a 2-connected chordless graph, then either $H$ is 2-sparse or $H$ admits a proper 2-cutset.
\end{theorem}

Given a graph $H$, a {\em subdivision} of $H$ is a graph obtained from $H$ by replacing each
edge $uv$ of $H$ by a $(u,v)$-path of length at least one in such a way that none of 
the paths that we add has an internal vertex in $V(H)$ or in one of the other paths we add.

A \emph{model} of a graph $H$ in a graph $G$ is a collection $(A_h)_{h \in V(H)}$ of disjoint subsets of $V(G)$ such that $G[A_h]$ is connected for all 
$h \in V(H)$, and for every edge $e = hh' \in E(H)$, there is at least one edge between $A_h$ and $A_{h'}$ in $G$. 
We say that the graph $G$ contains $H$ \emph{as a minor} (or \emph{contains an $H$-minor}) if $G$ contains a model of $H$.
The following is well known (see, for example, \cite{diestel.2017.book}, 1.7).

\begin{lemma}
\label{lem.minor.degree3}
Let $G$ and $H$ be graphs such that $\Delta(H) \leq 3$ and $G$ contains $H$ as a minor. Then $G$ has a subdivision of $H$ as a subgraph.
\end{lemma}

\begin{lemma}[\cite{duffin1965topology}]
 \label{lem:sp}
Let $G$ be a graph which does not contain a $K_4$-minor. Then $G$ contains a vertex of degree at most two. 
\end{lemma}



\bigskip
We are now ready to prove \autoref{th.main.reduction}.

\begin{proof}[Proof of \autoref{th.main.reduction}]
Let us assume towards a contradiction that the theorem does not hold. Let $G$ be a minimal, with respect to the number of its vertices, counterexample for the theorem.

\begin{claim}
\label{cl.counterex.reduction}
Let $\{u,v\} \in E(G)$. Then $d(u) \geq 3$ or $d(v) \geq 3$. 
\end{claim}

\begin{proof}[Proof of \autoref{cl.counterex.reduction}]
Let $\{u,v\} \in E(G)$. Since $G$ is 2-connected it follows that $d(u)\geq 2$ and $d(v)\geq 2$. Let us suppose towards a contradiction that $d(u)=d(v)=2$
and let $u^{\prime}$ be the unique vertex of the set $N_{G}(u)\setminus \{v\}$ and let $v^{\prime}$ be the unique vertex of the set $N_{G}(v)\setminus \{u\}$.

Let $\tilde{G}$ be the graph obtained from $G$ by contracting the edge $\{u,v\}$, and let $v_{uv} \in V(\tilde{G})$ be the vertex formed by the contraction.

We claim that $\tilde{G}$ is 2-connected. Suppose not, and let $w\in V(\tilde{G})$ be a cutvertex of $\tilde{G}$. If $w \in V(\tilde{G})\setminus \{v_{uv}\}$, 
then $w$ is a cutvertex of $G$, which is a contradiction. Thus, $w = v_{uv}$ and hence both $v$ and $u$, are cutvertices of $G$ which is again a contradiction.

We claim that $\tilde{G}$ has no Hamiltonian cycle. Suppose not, and let $C$ be a Hamiltonian cycle in $\tilde{G}$.
Let $C^{\prime}:= C \setminus v_{uv}$ and let $C^{\prime \prime}: = (V(C^{\prime})\cup \{u,v\}, E(C^{\prime})\cup \{ \{u,u^{\prime}\}, \{u, v\}, \{v,v^{\prime}\} \})$.
Then, $C^{\prime \prime}$ is a Hamiltonian cycle in $G$, which is a contradiction. Thus $\tilde{G}$ has no Hamiltonian cycle.

We claim that $\tilde{G}$ has no induced subgraph which is isomorphic to a 3PC. Suppose not, and let 
$H$ be such an induced subgraph of $\tilde{G}$. Since by our assumptions $H$ is not an induced subgraph of $G$, it follows that $v_{uv} \in V(H)$. 
Since $v_{uv}$ has degree two in $\tilde{G}$, and since $H$ is 2-connected, we deduce that $v_{uv}$ has degree two in $H$. Thus, $v_{uv}$ is an internal vertex
of a path of length at least two of $H$.

Let $H^{\prime}:= H \setminus v_{uv}$ and let $H^{\prime \prime} := (V(H^{\prime})\cup \{u,v\}, E(H^{\prime})\cup \{ \{u,u^{\prime}\}, \{u, v\}, \{v,v^{\prime}\}\} )$.
Then, $H^{\prime \prime}$ is an induced subgraph of $G$ which is isomorphic to a 3PC,
which is a contradiction. Hence $\tilde{G}$ has no induced subgraph isomorphic to a 3PC.

By the above it follows that $\tilde{G}$ is a counterexample for \autoref{th.main.reduction}, which is a contradiction because
$|V(\tilde{G})| = |V(G)|-1 < |V(G)|$, and $G$ is a minimum counterexample.
\end{proof}

\begin{claim}
\label{cl.counterex.2cutset.one.small}
Let $\{u,v\}$ be a 2-cutset of $G$. Then there exist induced subgraphs of $G$, say $G_{1}, G_{2}$, 
such that $|V(G_{1})|, |V(G_{2})| < |V(G)|$, $V(G_{1}) \cap V(G_{2}) = \{u,v\}$, 
$V(G_{1}) \cup V(G_{2}) = G$, $G_{1}$ is a $(u,v)$-path of length two or a triangle and $G_{2}$ is neither a path nor a cycle.
\end{claim}

\begin{proof}[Proof of \autoref{cl.counterex.2cutset.one.small}]
We first prove that $G \setminus \{u,v\}$ has exactly two connected connected components. Suppose not, and let 
$G_{1}, G_{2}$ and $G_{3}$ be three connected components of $G \setminus \{u,v\}$. Since $G$ is 2-connected, neither of the vertices 
$u,v$ is a cutvertex, and thus for each $i \in [3]$, 
each of $u,v$ has at least one neighbor in $G_{i}$. 
For each $i \in [3]$, let $P_{i}$ be a shortest $(u,v)$-path with all its internal vertices in $G_{i}$ and such that $P_{i}$ does not use the edge $\{u,v\}$ if it is present in $G$. 
$P_{i}$ has length at least two, and $G[V(P_{1}) \cup V(P_{2}) \cup V(P_{3})]$ is either a theta$^{+}$ or a theta (depending on whether $\{u,v\}\in E(G)$ or not), 
which is a contradiction. Thus $G \setminus \{u,v\}$ has exactly two connected components.

Let $H_{1}, H_{2}$ be the two connected components of $G\setminus \{u,v\}$ and let 
$H^{\prime}_{1}: = G[V(H_{1}) \cup \{u,v\}]$ and $H^{\prime}_{2}: = G[V(H_{2}) \cup \{u,v\}]$.

We claim that there exists $i \in [2]$, such that the graph $H^{\prime}_{i}$ is neither a $(u,v)$-path, nor a cycle.
Suppose not, then either $\{u,v\} \notin E(G)$ and the graph $G$ is a cycle, and hence a Hamiltonian graph; or $\{u,v\} \in E(G)$, in which case
the subgraph of $G$ which we obtain if we delete the edge 
$\{u,v\}$ is a Hamiltonian cycle of $G$. Both outcomes contradict the fact that $G$ is a non-Hamiltonian graph.

We claim that there exists $i \in [2]$, such that the graph $H^{\prime}_{i}$ is either a $(u,v)$-path or a cycle.
Suppose not. For each $i \in [2]$ let $P_{i}$ be a shortest $(u,v)$-path in $H^{\prime}_{i}$ which does not use the edge $\{u,v\}$ if it is present in $G$. 
Then $|V(P_{i})| < |H^{\prime}_{i}|$.
The graphs $H^{\prime \prime}_{1}: = G[V(H^{\prime}_{1})\cup V(P_{2})]$ and $H^{\prime \prime}_{2}: =G[V(H^{\prime}_{2})\cup V(P_{1})]$ 
are 2-connected induced subgraphs of $G$ and
$|V(H^{\prime \prime}_{1})| < |V(G)|$ and $|V(H^{\prime \prime}_{2})| < |V(G)|$. Thus, by the minimality of $G$ as a counterexample, it follows that for each $i \in [2]$, 
the graph $H^{\prime \prime}_{i}$ has a Hamiltonian cycle $C_{i}$ since every vertex in $V(P_{3-i})\setminus \{u,v\}$ has degree $2$ in $H^{\prime \prime}_{i}$, it follows that
$E(P_{3-i}) \subseteq E(C_{i})$ for $i \in \{1,2\}$.
Let $C: = (C_{1}\setminus P_{2})\cup (C_{2}\setminus P_{1})$. Then $C$
is a Hamiltonian cycle of $G$, which is a contradiction. 

Without loss of generality we may assume that $H^{\prime}_{1}$ is either a $(u,v)$-path, or a cycle and that $H^{\prime}_{2}$ is neither a path nor a cycle.
Since $\{u,v\}$ is a cutset, we have that $|V(H_{1}^{\prime})| \geq 3$.
We claim that $|V(H_{1}^{\prime})| = 3$. Suppose not. Then there exist $v_{1},v_{2} \in V(G_{1})$, such that $\{v_{1}, v_{2}\}\in E(G)$ and 
$d(v_{1}) = d(v_{2}) = 2$, which contradicts \autoref{cl.counterex.reduction}. 
The graphs $G_{1}$ and $G_{2}$ are induced subgraphs of $G$, which witness that the claim holds.
\end{proof}

\begin{corollary}
\label{cl.reduction.special.edges}
If $\{u,v\}$ is a clique cutset of size two of $G$, then
there exist induced subgraphs of $G$, say $G_{1}, G_{2}$, such that $|V(G_{1})|, |V(G_{2})| < |V(G)|$, $V(G_{1}) \cap V(G_{2}) = \{u,v\}$, 
$V(G_{1}) \cup V(G_{2}) = V(G)$, $G_{1}$ is a triangle and $G_{2}$ is neither a path nor a cycle.
In this case we call the edge $\{u,v\}$ a {\em special edge} of the graph $G$, 
and the unique $(u,v)$-path of length two in $G_{1}$ the {\em corresponding path} of this special edge.
\end{corollary}

\begin{proof}[Proof of \autoref{cl.reduction.special.edges}]
Immediate by \autoref{cl.counterex.2cutset.one.small}.
\end{proof}

\bigskip
Let $G^{\prime}$ be the induced subgraph of $G$ that we obtain if for each special edge of $G$, we 
delete the unique internal vertex of its corresponding path.
In claims \ref {cl.cutsets.in.G'}, \ref{cl.G'.no.clique.cutset}, \ref{cl.G'.2cutset.one.side.small} and \ref{cl.G'.no.HC.with.special.edges}, we prove 
some properties of the graph $G^{\prime}$.

\begin{claim}
\label{cl.cutsets.in.G'}
Let $X$ be a cutset in $G^{\prime}$. Then $X$ is a cutset in $G$.
\end{claim}

\begin{proof}[Proof of \autoref{cl.cutsets.in.G'}]
Let us suppose towards a contradiction that $G\setminus X$ is connected. 
Let $G^{\prime}_{1}, G^{\prime}_{2}$ be two connected components of $G^{\prime} \setminus X$.
Let $u\in V(G^{\prime}_{1})$ and $v\in V(G^{\prime}_{2})$. Let $P$ be an induced $(u,v)$-path in $G$.
Let $x \in V(P)\setminus V(G^{\prime})$. Then $N_{V(G)}(x)$ is a clique, so $|N_{V(P)}(x)|\leq 1$ as $P$ is induced. 
But then $x \in \{u,v\}$, a contradiction.
\end{proof}

\begin{claim}
\label{cl.G'.no.clique.cutset}
The graph $G^{\prime}$ does not admit a clique cutset.
\end{claim}

\begin{proof}[Proof of \autoref{cl.G'.no.clique.cutset}]
By \autoref{cl.cutsets.in.G'}, the definition of $G^{\prime}$, and the fact that $G$ is 2-connected, we conclude that $G^{\prime}$ 
has neither cutvertex, nor clique cutset of size two.
Also, since $G^{\prime}$ is wheel-free, it follows that it does not contain an induced subgraph isomorphic to 
the complete graph on $4$ vertices and thus $G^{\prime}$ has no clique cutset of size four or more.

We claim that $G^{\prime}$ has no clique cutset of size three. Suppose not and let $\{u,v,w\}$ be a clique cutset of $G^{\prime}$.
Let  $G_{1}^{\prime}, G_{2}^{\prime}$ be two connected components of $G^{\prime} \setminus \{u,v,w\}$.
Since $G^{\prime}$ has no clique cutset of size two, it follows that for each $i \in [2]$ each of $u,v$ and $w$, has at least one neighbor in $G_{i}^{\prime}$.
For each $i \in [2]$, let $P_{i}$ be a shortest $(u,v)$-path which does not use the edge $\{u,v\}$, and with all its internal vertices in $G_{i}^{\prime}$.
Then $G[V(P_{i})]$ is a cycle and $w \notin V(P_{i})$. 

We claim that for each $i \in [2]$ no internal vertex of $P_{i}$ is adjacent with the vertex $w$. Suppose not. Then $G^{\prime}[V(P_{i}) \cup \{u,v,w\}]$ is a wheel, which is a contradiction.
It follows that $G^{\prime}[V(P_{1}) \cup V(P_{2}) \cup \{u,v,w\}]$, is a $theta^{+}$, which is a contradiction.
Thus, $G^{\prime}$ has no clique cutset of size three, and hence $G^{\prime}$ has no clique cutset.
\end{proof}

\begin{claim}
\label{cl.G'.2cutset.one.side.small}
Let $\{u,v\}$ be a 2-cutset of $G^{\prime}$. Then $G^{\prime} \setminus \{u,v\}$ has exactly two connected components, and one of these is a single vertex.
\end{claim}

\begin{proof}[Proof of \autoref{cl.G'.2cutset.one.side.small}]
Immediate by \autoref{cl.counterex.2cutset.one.small} and \autoref{cl.cutsets.in.G'}.
\end{proof}

\begin{claim}
\label{cl.G'.no.HC.with.special.edges}
The graph $G^{\prime}$ does not have a Hamiltonian cycle which contains all the special edges of $G$.
\end{claim}

\begin{proof}[Proof of \autoref{cl.G'.no.HC.with.special.edges}]
Let us assume towards a contradiction that $G^{\prime}$ has a Hamiltonian cycle, say $C^{\prime}$, which contains all the special edges of $G$. 
For every special edge $\{x,y\}$ of $G$, let us denote by $P_{xy}$ its corresponding path. Let $C$ be the subgraph of $G$ obtained from $C^{\prime}$ as follows:
For every special edge $\{x,y\}$ of $G$, we delete from $C^{\prime}$ the edge $\{x,y\}$ and we add the path $P_{xy}$. Then, $C$ is a Hamiltonian cycle of $G$,
which contradicts the fact that $G$ is a non-Hamiltonian graph.
\end{proof}

\begin{claim}
\label{cl.line.graph.of.chordless}
There exists a triangle-free chordless graph $H$, with $\Delta(H) \leq 3$, such that $G^{\prime} = L(H)$.
\end{claim}

\begin{proof}[Proof of \autoref{cl.line.graph.of.chordless}]
By our assumptions for $G$, and the definition of $G^{\prime}$, the graph $G^{\prime}$ 
is an only-prism graph. By \autoref{cl.G'.no.clique.cutset}, $G^{\prime}$ has no clique cutset.
Thus, by \autoref{th.only.prism.str}, there exists a triangle-free chordless graph $H$, such that $G^{\prime} = L(H)$.
Since $G^{\prime}$ is $K_{4}$-free, it follows that $\Delta(H) \leq 3$.
\end{proof}

\bigskip
In Claims 8--12, we prove some properties of the graph $H$.

\begin{claim}
\label{cl.H.is.2conn}
$H$ is 2-connected.
\end{claim}

\begin{proof}[Proof of \autoref{cl.H.is.2conn}]
Since $G^{\prime}$ is connected and $G^{\prime} = L(H)$, it follows that $H$ is connected. 
Let us assume towards a contradiction that $H$ is not 2-connected. Let $v \in V(H)$ be a cutvertex of $H$.
If every edge of $H$ is incident with $v$, then $G^{\prime}$ is a complete graph. Since $G^{\prime}$ is $2$-connected and $K_{4}$-free, 
it follows that $G^{\prime}$ is isomorphic to $K_{3}$, contrary to \autoref{cl.G'.no.HC.with.special.edges}.
Now let $X$ be a component of $H\setminus \{v\}$ containing an edge $e$ not incident with $v$. Let $K$ be the set of vertices in $G^{\prime}$
that correspond to edges of $H$ of the form $\{u,v\}$ where $u \in X$. Then $K$ is a clique cutset in $G^{\prime}$, contrary to \autoref{cl.G'.no.clique.cutset}.
\end{proof}

\begin{claim}
\label{cl.H.has.no.two.edgecutset.with.two.sides.big}
If $F$ is a 2-edge-cutset in $H$, then the graph $H \setminus F$ has exactly two connected components, one of which is either a single vertex or a single edge.
\end{claim}

\begin{proof}[Proof of \autoref{cl.H.has.no.two.edgecutset.with.two.sides.big}]
Follows immediately by \autoref{cl.G'.2cutset.one.side.small} and by the fact that $G^{\prime} = L(H)$.
\end{proof}

\begin{claim}
\label{cl.H.2sparse}
The graph $H$ does not admit a proper 2-cutset, and  $H$ is 2-sparse. 
\end{claim}

\begin{proof}[Proof of \autoref{cl.H.2sparse}]
By \autoref{th.structure.biconn.chordless}, it suffices to prove that $H$ does not admit a proper 2-cutset. Suppose for a contradiction that $H$ admits a proper 2-cutset, and let $(X,Y,u,v)$ be
a split of this proper 2-cutset. Since $u$ and $v$ each have degree at most three in $H$ by \autoref{cl.line.graph.of.chordless}, it follows that there is a vertex $u' \in V(H)$ such that either $N(u) \cap X = \{u'\}$ or $N(u) \cap Y = \{u'\}$. Likewise, there is a vertex $v' \in V(H)$ such that either  $N(v) \cap X = \{v'\}$ or $N(v) \cap Y = \{v'\}$. It follows that $\{uu', vv'\}$ is a two-edge cutset in $H$. By \autoref{cl.H.has.no.two.edgecutset.with.two.sides.big}, it follows that $H \setminus \{uu', vv'\}$ has exactly two connected components, say $X'$ and $Y'$, and that $X'$ is a clique and $|X'| \leq 2$. Since $X, Y \neq \emptyset$ and $H \setminus \{uu', vv'\}$ contains no $(X, Y)$-path, we may assume that $X \subseteq X'$ and $Y \subseteq Y'$. It follows that $|X| \leq 2$. If $u$ has two neighbors in $X$, then $u' \in Y$ and hence $X'$ contains $u$ and its two neighbors, contrary to the fact that $|X'| \leq 2$; and likewise for $v$. It follows that $u, v$ both have exactly one neighbor in $X$. If $u$ and $v$ have a common neighbor $x$ in $X$, then either $|X| = 1$ and $H[X \cup \{u, v\}]$ is a path, contrary to fact that $\{u, v\}$ is a proper 2-cutset; or $x$ is a cutvertex of $H$, contrary to \autoref{cl.H.is.2conn}. It follows that $|X| = 2$, and so $X = X'$ and $u', v' \in X$. But now $H[X \cup \{u, v\}]$ is a four-vertex path with vertices $u, u', v', v$ in this order, contrary to the fact that $\{u, v\}$ is a proper 2-cutset. 
\end{proof}

\begin{claim}
\label{cl.H.noK4.minor}
The graph $H$ has no $K_{4}$-minor.
\end{claim}

\begin{proof}[Proof of \autoref{cl.H.noK4.minor}]
Let us assume towards a contradiction that $H$ has a $K_{4}$-minor. Then, by \autoref{lem.minor.degree3}, it follows that $H$ has a subgraph $K$ which is a subdivision of $K_{4}$.
Thus, there exist four vertices of $K$, say $\{v_{1}, v_{2}, v_{3}, v_{4}\}$, and a set of six internally vertex-disjoint paths 
$$\mathcal{P}: = \{P_{i,j}: i,j\in [4] \text{ and } i<j \text{ and } P_{i,j} \text{ is a } (v_{i}, v_{j})\text{-path in } K\},$$
such that $K = \bigcup \mathcal{P}$. 

Since $H$ is 2-sparse, it follows that each path in $\mathcal{P}$ has length at least two.
We claim that each path in $\mathcal{P}$ has length exactly two. Suppose not. Without loss of generality, we may assume that 
$P_{1,2}$ has length at least three. Let $J$ be the subgraph of $K$ (and thus of $H$) obtained by deleting from $K$ the internal vertices of the path $P_{3,4}$.
Then, the graph $L(J)$ is a prism. By \autoref{lem.line.graph.subgraph}, since $J$ is a subgraph of $H$ the graph $L(J)$ is an induced subgraph of $G^{\prime}$ 
and thus $G^{\prime}$ has a prism as an induced subgraph, a contradiction. 
Thus, each path in $\mathcal{P}$ has length exactly two.

We claim that $H=K$. Suppose not. Then, since $H$ is connected, there exists an edge in $E(H)\setminus E(K)$ which is incident to a vertex $v \in V(K)$. 
Since $\Delta(H) \leq 3$, we have that $v \notin \{v_{1}, v_{2}, v_{3}, v_{4}\}$. 
Thus there exists a path $P \in \mathcal{P}$ such that $v$ is the unique internal vertex of $P$. Then, $d_{H}(v) = 3$,
and thus both the two edges of $P$ which are incident with the vertex $v$ are not incident with a vertex of degree at most two, which contradicts that, by \autoref{cl.H.2sparse}, 
the graph $H$ is 2-sparse. Hence, $H=K$ and thus $G^{\prime} = L(K)$.

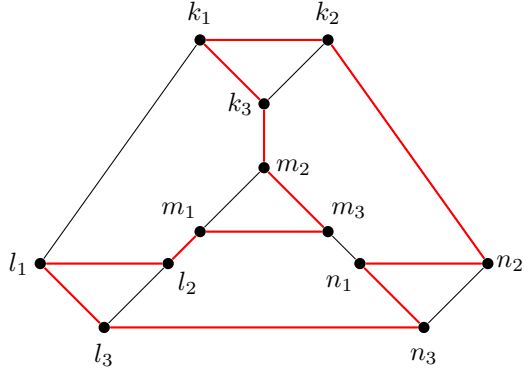
\begin{figure}[h!]\centering

 \begin{tikzpicture}
	\node [label={[label distance=0.05cm]above:$k_{1}$}] (k1) at (-1,2.5) {};
	\node [label={[label distance=0.05cm]above:$k_{2}$}] (k2) at (1,2.5) {};
	\node [label=left:{$k_{3}$}] (k3) at (0,1.5) {};
	
	\node [label={[label distance=0.05cm]right:$m_{2}$}] (m2) at (0,0.5) {};
	\node [label={[label distance=0.05cm]135:$m_{1}$}] (m1) at (-1,-0.5) {};
	\node [label={[label distance=0.05cm]45:$m_{3}$}] (m3) at (1,-0.5) {};
			
    \node [label={[label distance=0.08cm]-45:$l_{2}$}] (l2) at (-1.5,-1) {};
	\node [label=left:{$l_{1}$}] (l1) at (-3.5,-1) {};
	\node [label={[label distance=0.08cm]below:$l_{3}$}] (l3) at (-2.5,-2) {};
			
	\node [label={[label distance=0.08cm]225:$n_{1}$}] (n1) at (1.5,-1) {};
	\node [label=right:{$n_{2}$}] (n2) at (3.5,-1) {};
	\node [label={[label distance=0.08cm]below:$n_{3}$}] (n3) at (2.5,-2) {};
		
    \draw (k1) -- (k2);
	\draw (k1) -- (k3);
	\draw (k2) -- (k3);
			
	\draw (m1) -- (m2);
	\draw (m1) -- (m3);
	\draw (m2) -- (m3);
			
	\draw (n1) -- (n2);
	\draw (n1) -- (n3);
	\draw (n2) -- (n3);
			
	\draw (l1) -- (l2);
	\draw (l1) -- (l3);
	\draw (l2) -- (l3);
			
    \draw (k3) -- (m2);
	\draw (l2) -- (m1);
	\draw (n1) -- (m3);
	\draw (k1) -- (l1);
	\draw (l3) -- (n3);
	\draw (n2) -- (k2);
			
	\draw [red, thick] (l1)--(l3)--(n3)--(n1)--(n2)--(k2)--(k1)--(k3)--(m2)--(m3)--(m1)--(l2)--(l1);
	\end{tikzpicture}
    \caption{Proof of \autoref{cl.H.noK4.minor}: The Hamiltonian graph $L(K)$.}
    \label{fig.Hamiltonian.K}
\end{figure}

Let $\{k_{1}, k_{2}, k_{3}\}$, $\{l_{1}, l_{2}, l_{3}\}$, $\{m_{1}, m_{2}, m_{3}\}$ and $\{n_{1}, n_{2}, n_{3}\}$,
be the four cliques of size three of the graph $G^{\prime}$. Without loss of generality we may assume that:
$$\{\{k_{3},m_{2}\}, \{m_{1}, l_{2}\}, \{m_{3}, n_{1}\}, \{k_{1}, l_{1}\}, \{l_{3}, n_{3}\}, \{n_{2}, k_{2}\} \} \subseteq E(G^{\prime}).$$
We claim that $G$ has no special edges. Suppose not. Let $e$ be a special edge of $G$ and let $P_{e}$ be the corresponding path of the special edge $e$.
Either both the endpoints of $e$ lie in one of the four triangles of $G^{\prime}$
or the endpoints of $e$ lie in two different triangles of $G^{\prime}$. 

Suppose that the former case holds. Without loss of generality we may assume that $e = \{k_{1}, k_{2}\}$. 
Then $G[V(P_{e})\cup \{k_{3}, l_{1}, l_{3}, n_{3}, n_{2}\} ]$ is a theta$^{+}$, which is a contradiction.
Hence, there exist two triangles of $G^{\prime}$ such that special edge $e$ joins these two triangles. Without loss of generality we may assume that $e = \{k_{1}, l_{1}\}$. 
Then $G[V(P_{e})\cup \{k_{1}, k_{2}, k_{3}\} \cup \{l_{1}, l_{2}, l_{3}\} \cup \{m_{1}, m_{2}, n_{3}, n_{2}\}]$ is a prism$^{+}$, which is a contradiction. 

Thus $G$ has no special edges and hence $G=G^{\prime} = L(K)$. Let $C: = l_{1}l_{3}n_{3}n_{1}n_{2}k_{2}k_{1}k_{3}m_{2}m_{3}m_{1}l_{2}l_{1}$ (see \autoref{fig.Hamiltonian.K}).
Then $C$ is a Hamiltonian cycle in $G$, which is a contradiction. Hence, $H$ has no $K_{4}$-minor.
\end{proof}

\bigskip
Since $H$ is 2-connected, $G^{\prime} = L(H)$, and $G^{\prime}$ has no Hamiltonian cycle which uses all of its edges, it follows that 
$H$ has at least four vertices and $H$ contains a vertex of degree at least three. 
Thus, since by \autoref{cl.H.noK4.minor}, $H$ has no $K_{4}$-minor, it follows from \autoref{lem:sp} that $H$ has at least one vertex of degree two.
Let $c$ be such a vertex and let $e_{1}^{c}$ and $e_{2}^{c}$ be the two edges of $H$ which are incident to $c$. 
For each $i \in [2]$ let $P_{i}$ be a minimum-length path which has as one end the vertex $c$, contains the edge $e_{i}^{c}$, and the 
other end of $P_{i}$ is a vertex of degree three. 
The paths $P_{1}, P_{2}$ exist, because $H$ is 2-connected and not a cycle. Let $a,b$ be the degree-three ends of $P_{1}$ and $P_{2}$ respectively.

We claim that $a\neq b$. Suppose not. Then, since $P_{1}\cup P_{2}$ is a cycle, and $H$ is not a cycle, the vertex $a=b$ is a cutvertex of $H$, 
contradicting \autoref{cl.H.is.2conn}. Thus, $a\neq b$.

We claim that the sum of the lengths of $P_{1}$ and $P_{2}$ is at most three.  Suppose not, and let $e_{a}$ and $e_{b}$ be the edges 
of $P_{1}$ and $P_{2}$ which are incident to $a$ and $b$, respectively. Then $\{e_{a}, e_{b}\}$ is a 2-edge-cutset of $H$ and the graph 
$H\setminus \{e_{a}, e_{b}\}$ has at least two components, 
each of which contains at least two edges, contradicting \autoref{cl.H.has.no.two.edgecutset.with.two.sides.big}. This proves our claim. 

Let $\mathcal{C}$ be the set of the connected components of $H \setminus \{a, b\}$.
Since $\Delta(H) \leq 3$ and $H$ is 2-connected, it follows that either $|\mathcal{C}| = 2$ or $|\mathcal{C}| = 3$.

\begin{claim}
\label{cl.structure.of.H}
$|\mathcal{C}| = 3$.
\end{claim}

\begin{proof}[Proof of \autoref{cl.structure.of.H}]
Let us suppose towards a contradiction that $|\mathcal{C}| = 2$. Let $H_{1}, H_{2}$ be the two connected components of $H \setminus \{a,b\}$, 
where $c \in V(H_{1})$. 

Since $H$ is 2-sparse and the vertices $a,b$ have degree three in $H$, it follows that $\{a,b\} \notin E(H)$. Since $d(a) = d(b) = 3$ and $\{a,b\} \notin E(H)$, it follows that each of $a,b$ has two neighbors in $H_{2}$. Let $a_{1}, a_{2}$ and $b_{1}, b_{2}$ be the neighbors 
of $a$ and $b$ respectively in $H_{2}$. Let $A: = \{a_{1}, a_{2}\}$ and $B:= \{b_{1}, b_{2}\}$. 
By Menger's theorem (see, for example \cite{menger1927allgemeinen}), applied in the component $H_{2}$, either there exist two vertex-disjoint $(A,B)$-paths in $H_{2}$ or there exists a 
vertex separating $A$ from $B$ in $H_{2}$. 

We claim that there exist no two vertex-disjoint $(A,B)$-paths in $H_{2}$. Suppose not, and let 
$Q_{1}, Q_{2}$ be such paths. Since $H_{2}$ is connected, there exists a $(V(Q_{1}), V(Q_{2}))$-path. Let $Q_{3}$ be a minimum length $(V(Q_{1}), V(Q_{2}))$-path
in $H_{2}$. Then no internal vertex of $Q_{3}$ lies in $V(Q_{1})\cup V(Q_{2})$. Let
$$K: = H_{1} \cup Q_{1} \cup Q_{2} \cup Q_{3} + \{\{a,a_{1}\}, \{a,a_{2}\}, \{b, b_{1}\}, \{b, b_{2}\}\}.$$ 
Then $K$ is a subdivision of $K_{4}$, and thus $H$ has a $K_{4}$-minor, which contradicts \autoref{cl.H.noK4.minor}.

Hence, there exists $d \in V(H_{2})$, such that $d$ separates $A$ from $B$ in $H_{2}$. Let $H_{2}^{A}$ and $H_{2}^{B}$ be a partition of $H_{2} \setminus \{d\}$, 
where $A\setminus \{d\} \subseteq V(H_{2}^{A})$ and $B\setminus \{d\} \subseteq V(H_{2}^{B})$. Since $\Delta(H)\leq 3$, without loss of generality, we may assume 
that $d$ has exactly one neighbor in $H_{2}^{A}$. Let $d_{A}$ be the neighbor of $d$ in $H_{2}^{A}$. Let $F:= \{\{d,d_{A}\}, \{b,b^{\prime}\} \}$, where 
$b^{\prime}$ is the neighbor of $b$ in $P_{2}$. Then $H \setminus F$ has two connected components, 
say $F_{1}, F_{2}$, where $F_{1}$ contains $a$ and its neighbors, except possibly $d$, and $\{\{b, b_{1}\}, \{b, b_{2}\}\} \subseteq E(F_{2})$, which contradicts
\autoref{cl.H.has.no.two.edgecutset.with.two.sides.big}. 
Thus $|\mathcal{C}| = 3$.
\end{proof}

\begin{claim}
\label{cl.final.G.short.prism}
$G^{\prime}$ is a short prism. 
\end{claim}

\begin{proof}[Proof of \autoref{cl.final.G.short.prism}]
By \autoref{cl.structure.of.H}, we have that the graph $H \setminus \{a,b\}$ has three connected components. 
Let $C_{1}, C_{2}$ and $C_{3}$ be the three components of $H \setminus \{a,b\}$ where $c \in V(C_{1})$.
For each $i \in [3]$, let $e_{i}^{a}$ and $e_{i}^{b}$ be the edge of $H$ which has as one end the vertex $a$ or $b$, respectively, and 
its other end lies in $C_{i}$.
Since, by \autoref{cl.H.has.no.two.edgecutset.with.two.sides.big}, 
for each $i \in [3]$ there exists at least one component of $H \setminus \{e_{i}^{a}, e_{i}^{b}\}$ which is either a single vertex or a single edge, 
it follows that each of the components $C_{2}$ and $C_{3}$, is either a single vertex or a single edge. Recall that $C_1$ is a path by construction, shown to have length at most 1.

		\begin{figure}[h!]\centering
		
			\begin{tikzpicture}
			\node [label={[label distance=0.05cm]below:$b$}] (v1) at (0,0) {};
			\node (v2) at (0,0.75) {};
			\node (v3) at (1,0.75) {};
			\node (v4) at (-1,1.25) {};
			\node (v8) at (-1,0.75) {};
			\node (v5) at (0,1.25) {};
			\node (v6) at (1,1.25) {};
			\node [label={[label distance=0.05cm]above:$a$}] (v7) at (0,2) {};

			\draw (v1) -- (v2);
			\draw (v1) -- (v3);
		    \draw (v1) -- (v8);
		     \draw (v4) -- (v8);
		     \draw (v4) -- (v7);
			\draw (v2) -- (v5);
			\draw (v3) -- (v6);
			\draw (v7) -- (v5);
			\draw (v7) -- (v6);
			\end{tikzpicture}
			\hspace{1cm}
		\begin{tikzpicture}
			\node [label={[label distance=0.05cm]below:$b$}] (v1) at (0,0) {};
			\node (v2) at (0,0.75) {};
			\node (v3) at (1,0.75) {};
			\node (v4) at (-1,1) {};
			\node (v5) at (0,1.25) {};
			\node (v6) at (1,1.25) {};
			\node [label={[label distance=0.05cm]above:$a$}] (v7) at (0,2) {};

			\draw (v1) -- (v2);
			\draw (v1) -- (v3);
			\draw (v1) -- (v4);
			\draw (v2) -- (v5);
			\draw (v3) -- (v6);
			\draw (v7) -- (v4);
			\draw (v7) -- (v5);
			\draw (v7) -- (v6);
			\end{tikzpicture}
			\hspace{1cm}
			\begin{tikzpicture}
			\node [label={[label distance=0.05cm]below:$b$}] (v1) at (0,0) {};
			\node (v2) at (0,1) {};
			\node (v3) at (1,0.75) {};
			\node (v4) at (-1,1) {};
			\node (v6) at (1,1.25) {};
			\node [label={[label distance=0.05cm]above:$a$}] (v7) at (0,2) {};
			
			\draw (v1) -- (v2);
			\draw (v1) -- (v3);
			\draw (v1) -- (v4);
			\draw (v3) -- (v6);
			\draw (v7) -- (v4);
			\draw (v7) -- (v2);
			\draw (v7) -- (v6);
			\end{tikzpicture}
			\hspace{1cm}
		\begin{tikzpicture}
			\node [label={[label distance=0.05cm]below:$b$}] (v1) at (0,0) {};
			\node (v2) at (0,1) {};
			\node (v3) at (1,1) {};
			\node (v4) at (-1,1) {};
			\node [label={[label distance=0.05cm]above:$a$}] (v7) at (0,2) {};
			
			\draw (v1) -- (v2);
			\draw (v1) -- (v3);
			\draw (v1) -- (v4);
			\draw (v7) -- (v4);
			\draw (v7) -- (v2);
			\draw (v7) -- (v3);
			\end{tikzpicture}
			\caption{Proof of \autoref{cl.final.G.short.prism}: The graph $H$ is one of the graphs illustrated above.}
			\label{fig.four.cases.H}
		\end{figure}
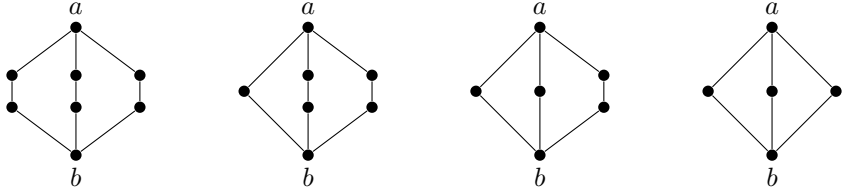

Since $H$ has no cutvertex, it follows that there are only four possibilities for the graph $H$, depending on how many of the components
$C_{1}, C_{2}, C_{3}$ consist of a single vertex (see \autoref{fig.four.cases.H}).

If each of $C_{1}, C_{2}, C_{3}$ contains two vertices, then $G^{\prime}$ is a prism, a contradiction. So at least one of the components $C_{1}, C_{2}, C_{3}$
contains only one vertex, and $G^{\prime}$ is a short prism.
\end{proof}

Since $G^{\prime}$ is a short prism, let $\{k_{1}, k_{2}, k_{3}\}, \{l_{1}, l_{2}, l_{3}\}, P_{1}, P_{2}, P_{3}$ as in the definition of the short prism in Section \ref{sec.intro}.

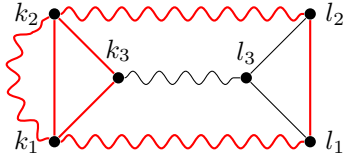
\begin{figure}[h!]\centering
	\begin{tikzpicture}
			\node [label={[label distance=0.05cm]left:$k_{1}$}] (k1) at (0,0) {};
			\node [label={[label distance=0.05cm]above:$k_{3}$}] (k3) at (1,1) {};
			\node [label={[label distance=0.05cm]left:$k_{2}$}] (k2) at (0,2) {};
			\node [label={[label distance=0.05cm]right:$l_{1}$}] (l1) at (4,0) {};
			\node [label={[label distance=0.05cm]above:$l_{3}$}] (l3) at (3,1)	{};
			\node [label={[label distance=0.05cm]right:$l_{2}$}] (l2) at (4,2) {};
			
			\draw[thick, red] (k1) -- (k3);
			\draw[thick, red] (k3) -- (k2);	
			\draw[thick, red] (k1) -- (k2);
			\draw (l1) -- (l3);
			\draw (l3) -- (l2);	
			\draw[thick, red] (l1) -- (l2);	
			
			\draw[thick, red, decorate, decoration=snake] (k2) -- (l2);
			\draw[decorate, decoration=snake] (k3) -- (l3);
			\draw[thick, red, decorate, decoration=snake] (k1) -- (l1);
			\draw[thick, red, decorate, decoration=snake] (k1) to[out=150,in=-150] (k2);
			\end{tikzpicture}
			\caption{A theta$^{+}$ in $G$ in case one of the triangles of $G^{\prime}$ contains a special edge.}
			\label{fig.theta.inG}
\end{figure}

We claim that no edge of a triangle of $G^{\prime}$ is a special edge of $G$. Suppose not. Without loss of generality we may assume that $\{k_{1}, k_{2}\} \in E(G^{\prime})$
is a special edge of $G$. Let $P_{k_{1}k_{2}}$ be the corresponding path of the edge $\{k_{1}, k_{2}\}$.
Then $G[V(P_{k_{1}k_{2}}) \cup V(P_{1})\cup V(P_{2}) \cup \{k_{3}\}]$ is a theta$^{+}$ (see \autoref{fig.theta.inG}), which contradicts the fact that $G$ is (theta$^{+}$)-free.

We claim that there exists at least one path in the set $\{P_{1}, P_{2}, P_{3}\}$ of length one 
such that the unique edge of this path is not a special edge of $G$.
Suppose not, 
let $Q\subseteq V(G)$ be the set of the internal vertices of the corresponding paths of all $P_i$ where $P_i$ is length one. 
Then, $G[Q \cup \{k_{1}, k_{2}, k_{3}\}\cup \{l_{1}, l_{2}, l_{3}\} \cup P_1 \cup P_2 \cup P_3]$ is a prism$^{+}$, which contradicts the fact that $G$ is (prism$^{+}$)-free.
Hence, without loss of generality we may assume that $P_{1}$ has length one and the unique edge of $P_{1}$ is not a special edge.

It follows that every special edge of $G^{\prime}$ is contained in $E(P_{2})\cup E(P_{3})$.
Let $C^{\prime}:= \left( P_{2} \cup P_{3} \cup (\{k_{1}, l_{1}\}, \emptyset)\right) + \{\{k_{1}, k_{2}\}, \{k_{1}, k_{3}\}, \{l_{1}, l_{2}\}, \{l_{1}, l_{3}\}\} $. 
Then, $C^{\prime}$ is a Hamiltonian cycle of $G^{\prime}$ which contains all the special edges of $G$, 
contradicting \autoref{cl.G'.no.HC.with.special.edges}. This concludes the proof.
\end{proof}

\end{document}